\documentclass[12pt,leqno]{article}
\usepackage{amssymb,amsfonts,amsmath,amsthm,amscd,mathrsfs}
\usepackage{PSFigure,psfrag}
\usepackage{graphicx}
\def\IE{{\mathbb E}}
\def\IP{{\mathbb P}}
\def\IR{{\mathbb R}}

\def\IZ{{\mathbb Z}}

\def\ar{a \rho + \frac{\rho^2}{2}}

\def\n{\noindent}

\def\wt{\widetilde}

\def\cB{{\cal B}}
\def\cC{{\cal C}}

\def\cI{{\cal I}}
\def\cJ{{\cal J}}

\def\cV{{\cal V}}

\dimendef\dimen=0

\newtheorem{theorem}{Theorem}[section]

\newtheorem{corollary}[theorem]{Corollary}

\newtheorem{remark}[theorem]{Remark}

\begin{document}

\noindent
~

\bigskip
\begin{center}
{\bf ON COUPLING AND VACANT SET LEVEL SET PERCOLATION}
\end{center}

\begin{center}
Alain-Sol Sznitman
\end{center}

\begin{center}
Preliminary Draft
\end{center}

\begin{abstract}
In this note we discuss vacant set level set percolation on a transient weighted graph. It interpolates between the percolation of the vacant set of random interlacements and the level set percolation of the Gaussian free field. We employ coupling and derive a stochastic domination from which we deduce in a rather general set-up a certain monotonicity property of the percolation function. In the case of regular trees this stochastic domination leads to a strict inequality between some eigenvalues related to Ornstein-Uhlenbeck semi-groups for which we have no direct analytical proof. It underpins a certain strict monotonicity property that has significant consequences for the percolation diagram. It is presently open whether a similar looking diagram holds in the case of $\IZ^d$, $d \ge 3$.
\end{abstract}

\vspace{4cm} 
\n
Departement Mathematik \hfill July 2018
\\
ETH Z\"urich\\
CH-8092 Z\"urich\\
Switzerland

\vfill

~
\newpage
\thispagestyle{empty}
~

\newpage
\setcounter{page}{1}

\setcounter{section}{-1}
\section{Introduction}

Given a connected, locally finite, transient weighted graph with countable vertex set $E$, one can consider the random interlacements $\cI^u$ at level $u \ge 0$ on $E$, with corresponding vacant set $\cV^u = E \backslash \cI^u$, and governing probability $\IP^I$, see \cite{Teix09b}, as well as the Gaussian free field $\varphi$ on $E$, with corresponding super level sets $\{\varphi > a\}$, $a \in \IR$, and governing probability $\IP^G$. The percolative properties of $\cV^u$ and of $\{\varphi > a\}$ have been the object of much interest, see for instance \cite{CernTeix12}, \cite{DrewPrevRodr}, \cite{DrewRathSapo14c}, and the references therein. There are strong links between the two models that result from Dynkin-type isomorphism theorems, see \cite{Szni12b}, \cite{Lupu16}, and in a broader context \cite{MarcRose06}, \cite{SaboTarr16}, \cite{LupuSaboTarr}. The consideration of cable graphs and the resulting extended couplings constructed in \cite{Lupu16}, later refined in \cite{Szni16}, provide in good cases efficient tools to compare the two percolation models. To date, they are the {\it only} tools to establish, for instance in the case of $\IZ^d$, $d \ge 3$, with unit weights, see \cite{Lupu16}, or for a broad class of transient trees, see
 \cite{AbacSzni18}, the inequality $h_* \le \sqrt{2u_*}$, where $h_*$ and $u_*$ respectively stand for the critical levels of level set percolation of the Gaussian free field, and of the percolation of the vacant set of random interlacements. The strict inequality $h_* < \sqrt{2 u_*}$ is more delicate and has so far only been established (with the help of the coupling of \cite{Szni16}) for a broad class of transient trees, see Section 3 of \cite{AbacSzni18}.
 
 In this note we consider the vacant set level set percolation that provides an interpolation between the two percolation models, and will perhaps bring some progress on the above mentioned issues. With 
$x_0$ some base point in $E$, we consider the probability under the product measure that the connected component of $x_0$ in $\cV^u \cap \{\varphi > a\}$ is infinite, namely:
\begin{equation}\label{0.1}
\tau(u,a) = \IP^I \otimes \IP^G \big[x_0 \overset{\cV^u \cap \{\varphi > \alpha\}}{\mbox{\Large $\longleftrightarrow$}} \infty\big], \;\mbox{for} \; u \ge 0, a \in \IR .
\end{equation}

\n
Thus $u = 0$ corresponds to level set percolation of the Gaussian free field, and (in an informal fashion) $a \rightarrow - \infty$ recovers vacant set percolation of random interlacements. Further, as explained in Remark \ref{rem1.3} 2), the positivity or the vanishing of $\tau(u,a)$ does not depend on the choice of the base point $x_0$.

In Theorem \ref{theo1.1} of Section 1 we obtain in a rather general set-up a key stochastic comparison result that dominates for $a \ge 0, \rho > 0$, the super level set $\{\varphi > a + \rho\}$ by the deletion from $\{\varphi > a\}$ of a certain random thickening of an independent interlacement $\cI^{a\rho + \frac{\rho^2}{2}}$, see (\ref{1.11}). The percolation function $\tau$ is non-increasing in $(u,a)$. However, as a simple minded by-product of Theorem \ref{theo1.1}, we show in Corollary \ref{cor1.2} that in the quite general set-up of Section 1
\begin{equation}\label{0.2}
\mbox{for any $h \ge 0$, the map $u \in \big [0,\frac{h^2}{2}\big] \rightarrow \tau (u, \sqrt{h^2 - 2u}) \ge 0$ is non-decreasing}.
\end{equation}

\n
This monotonicity property incidentally provides new sufficient conditions for $h_* < \sqrt{2 u_*}$ as explained in Remark in \ref{rem1.3} 3).

\medskip
We apply the techniques of Section 1 to the case of the $(d+1)$-regular tree $T$ (with $d \ge 2$) endowed with unit weights. The vacant cluster of  a base point $x_0$ for random interlacements on $T$ can in essence be studied by means of a Galton-Watson process, see Section 5 of \cite{Teix09b}, and this leads to a characterization of the critical value $u_*$ through the equation
\begin{equation}\label{0.3}
d \, e^{-u_* \frac{(d-1)^2}{d}} = 1.
\end{equation}

\n
The case of level set percolation of the Gaussian free field is definitely more intricate. Namely, to characterize the critical level $h_*$, one considers
\begin{equation}\label{0.4}
\mbox{$\nu$ the centered Gaussian law on $\IR$ with variance $\sigma^2 \stackrel{\rm def}{=} \frac{d}{d^2 -1}$},
\end{equation}
and for $h \in \IR$ one defines
\begin{equation}\label{0.5}
\begin{split}
\lambda_h = &\; \mbox{the operator norm in $L^2(\nu)$ of the self-adjoint Hilbert-Schmidt operator}
\\[-0.5ex]
& \; \mbox{$L_h = \pi_h \, L \, \pi_h$, where $\pi_h$ is the multiplication by $1_{[h,+ \infty)}$, and $L = d \; Q_{t =  \frac{1}{d}}$,}
\end{split}
\end{equation}

\n
with $Q_t, t \ge 0$, the Ornstein-Uhlenbeck semi-group  with invariant measure $\nu$ given by $(Q_t \,g)(a) = E^Y[g(a e^{-t} + \sqrt{1-e^{-2t}} \,Y)]$ for $a \in \IR$, $g$ in $L^2(\nu)$, and $E^Y$ the expectation relative to a $\nu$-distributed random variable $Y$. As shown in Section 3 of \cite{Szni16}
\begin{align}
&\mbox{$h \rightarrow \lambda_h$ is a decreasing homeomorphism of $\IR$ onto $(0,d)$, and} \label{0.6}
\\[0.5ex]
&\mbox{$h_*$ is characterized by the identity $\lambda_{h_*} = 1$}. \label{0.7}
\end{align}

\n
One knows that $0<h_* < \sqrt{2 u_*}$, see Corollary 4.5 of \cite{Szni16}, but no explicit formula for $h_*$ (i.e. comparable to (\ref{0.3})) is presently available.

\medskip
In Section 2 we use in an essential way the stochastic domination established in Theorem \ref{theo1.1} to prove a key strict inequality stated in Theorem \ref{theo2.1}. It leads to Corollary \ref{cor2.2}, where it is shown that the quantity 
\begin{equation}\label{0.8} 
\lambda (u,a) = \lambda_a \,e^{-u \frac{(d-1)^2}{d}}, \; \mbox{for $u \ge 0, a \in \IR$},
\end{equation}
which is decreasing in $(u,a)$, however satisfies the strict monotonicity property
\begin{equation}\label{0.9}
\mbox{for any $h > 0$, $u \in \big[0,\frac{h^2}{2}\big] \longrightarrow \lambda \big(u, \sqrt{h^2 - 2u}\big)$ is strictly increasing}.
\end{equation}

\n
In Theorem \ref{theo3.1} of Section 3 we present the link relating the functions $\tau$ and $\lambda$ of (\ref{0.1}) and (\ref{0.8}). This endows the curve $\lambda(u,a) = 1$ with the interpretation of a ``critical line'' for vacant set level set percolation on $T$, generalizing in spirit (\ref{0.7}). The strict monotonicity property (\ref{0.9}) established in Corollary \ref{cor2.2} then provides significant information on the loci $\tau = 0$ and $\tau > 0$ stated in Corollary \ref{cor3.2}, see also Figure 1. It is presently open, whether a similar looking diagram holds in the case of $\IZ^d$, $d \ge 3$. 

\section{Coupling and stochastic domination}
\setcounter{equation}{0}

In this section we first introduce some objects related to the cable graph attached to our basic transient graph and recall some of their properties. In the main Theorem \ref{theo1.1} of this section, we provide in a quite general set-up a coupling that relates level sets above some positive value of the Gaussian free field in the cable graph to the random interlacements on the cable graph. This comes as an application of the coupling constructed in Section 2 of \cite{Szni16} that refines Lupu's coupling from \cite{Lupu16}. The stochastic domination is then a natural consequence. In Section 2 it will be a key tool when we discuss the case of regular trees. As a simple minded application of the stochastic domination in Theorem \ref{theo1.1}, we establish in Corollary \ref{cor1.2} the monotonicity property of the function $\tau$ stated in (\ref{0.2}).

We thus consider a locally finite, connected, transient weighted graph with vertex set $E$, and symmetric weights $c_{x,y} = c_{y,x} \ge 0$, which are positive exactly when $x \sim y$, that is, when $x$ and $y$ are neighbors. The Green function can then be constructed by considering the discrete time walk on $E$ that jumps, when in $x \in E$, to a neighbor $y$ of $x$ with the probability $c_{x,y} / \lambda_x$ where $\lambda_x = \sum_{x' \sim x} c_{x,x'}$. If $P_x$ denotes the canonical law of the walk starting at $x$, $E_x$ the corresponding expectation, and $(X_k)_{k \ge 0}$ the canonical process, the Green function is symmetric and equals
\begin{equation}\label{1.1}
g(x,y) = \frac{1}{\lambda_y} \; E_x \big[\textstyle\sum\limits^\infty_{k=0} 1\{X_k = y\}\big], \; \mbox{for $x,y \in E$}.
\end{equation}

\n
The cable graph is obtained by attaching to each edge $e = \{x,y\}$ of the above graph a compact interval of length $(2c_{x,y})^{-1}$, with endpoints respectively identified to $x$ and $y$, and denoted by $\wt{[x,y]}$. We denote by $\wt{E} \supseteq E$ the resulting set, which is endowed with a natural metric and Lebesgue measure $m$, and by $\wt{g}(z,z')$, $z,z' \in \wt{E}$   the Green function on the cable graph. It is jointly continuous, and its restriction to $E \times E$ coincides with $g$ in (\ref{1.1}). We refer to Section 2 of \cite{Lupu16}, Section 2 of \cite{Folz14}, and Section 3 of \cite{Zhai18} for more details.

The Gaussian free field on the cable graph $\wt{E}$ can be realized as the law $\wt{\IP}^G$ on the canonical space $\wt{\Omega}$ of continuous real valued functions on $\wt{E}$, endowed with its canonical $\sigma$-algebra generated by the canonical coordinates $(\wt{\varphi}_z)_{z \in \wt{E}}$, such that
\begin{equation}\label{1.2}
\begin{array}{l}
\mbox{under $\wt{\IP}^G$, $(\wt{\varphi}_z)_{z \in \wt{E}}$ is a centered Gaussian process with covariance}
\\
\mbox{$\wt{E}^G [\wt{\varphi}_z \wt{\varphi}_{z'}] = \wt{g}(z,z')$, for $z,z' \in \wt{E}$ (with $\wt{\IE}^G$ the $\wt{\IP}^G$-expectation)}.
\end{array}
\end{equation}

\n
The law of $(\wt{\varphi}_x)_{x \in E}$ (i.e.~the restriction of $\wt{\varphi}$ to $E$) then coincides with the canonical law $\IP^G$ (on $\IR^E$) of the Gaussian free field on $E$ (i.e.~centered Gaussian with covariance $g(\cdot, \cdot)$ from (\ref{1.1})). As a consequence of the Markov property satisfied by $\wt{\varphi}$ under $\wt{\IP}^G$, see for instance (1.8) of \cite{Szni16}, one also knows that (see above Lemma 2.1 of \cite{DrewPrevRodr}):
\begin{equation}\label{1.3}
\begin{array}{l}
\mbox{Under $\wt{\IP}^G$, conditionally on $(\wt{\varphi}_x)_{x \in E}$, the law of the restrictions of $\wt{\varphi}$ to}
\\
\mbox{the various segments $[\wt{x,y}]$ of $\wt{E}$ between neighboring sites $x \sim y$ in $E$, is}
\\
\mbox{that of independent bridges in length $(2 c_{x,y})^{-1}$ of variance 2 Brownian}
\\
\mbox{motion, with boundary values $\wt{\varphi}_x, \wt{\varphi}_y$}.
\end{array}
\end{equation}

\n
We will assume from now on that
\begin{equation}\label{1.4}
\mbox{$\wt{\IP}^G$-a.s., $\{z \in \wt{E}$; $\wt{\varphi}_z >0]$ only has bounded connected components}.
\end{equation}

\n
This assumption is for instance satisfied in the case of the cable graph attached to $\IZ^d$, $d \ge 3$, endowed with unit weights, see Section 5 of \cite{Lupu16}, or in the case of the $(d+1)$-regular tree with $d \ge 2$, endowed with unit weights, see Section 4 of \cite{Szni16}. We also refer to Proposition 2.2 of \cite{AbacSzni18} for more general examples of transient trees where (\ref{1.4}) holds. We further assume that condition (1.43) of \cite{Szni16} holds, namely,
\begin{equation}\label{1.5}
\begin{array}{l}
\mbox{there is a $g_0 < \infty$ such that all connected components in $E$ of}
\\
\mbox{$\{x \in E; g(x,x) >g_0\}$ are finite}.
\end{array}
\end{equation}

\n
With (\ref{1.4}) and (\ref{1.5}) we will be able to use Theorem 2.4 of \cite{Szni16} (recalled below at the beginning of the proof of Theorem \ref{theo1.1}). It refines Lupu's coupling between the Gaussian free field and the random interlacements on the cable graph, see Proposition 6.3 of \cite{Lupu16}. 

We now proceed with some notation concerning random interlacements on the cable graph. Given a level $u > 0$, we consider on some probability space $(\wt{W}, \cB, \wt{\IP}^I)$ a continuous random field $(\wt{\ell}_{z,u})_{z \in \wt{E}}$ on $\wt{E}$ describing the local times with respect to the Lebesgue measure $m$ on $\wt{E}$ of random interlacements at level $u$ on $\wt{E}$, see Section 6 of \cite{Lupu16}, or below (1.21) of \cite{Szni16}, and assume that
\begin{align}
&\mbox{$\wt{\cI}^u \stackrel{\rm def}{=} \{z \in \wt{E}; \wt{\ell}_{z,u} >0\}$ is an open set, which only has unbounded components,} \label{1.6}
\\[1ex]
& \partial \wt{\cI}^u \cap E = \phi , \label{1.7}
\end{align}

\n
where for $A \subseteq \wt{E}$, $\partial A$ stands for the boundary of $A$.

\smallskip
The trace of $\wt{\cI}^u$ on $E$ then recovers random interlacements at level $u$ on $E$, in the sense that
\begin{equation}\label{1.8}
\mbox{$\wt{\cI}^u \cap E$ under $\wt{\IP}^I$ has same law as $\cI^u$ under $\IP^I$}.
\end{equation}

\n
Theorem \ref{theo1.1} below involves random open subsets of $\wt{E}$. We endow the collection of open subsets of $\wt{E}$ with the $\sigma$-algebra generated by the events $\{\wt{O}$ open in $\wt{E}$; $\wt{O} \supseteq \wt{K}\}$, where $\wt{K}$ varies over the collection of compact subsets of $\wt{E}$ (or equivalently over the collection of closed ``sub-intervals'' of segments $[\wt{x,y}]$, with $x \sim y$ in $E$). Incidentally, note that for any given $z$ in $\wt{E}$ the map that to $\wt{O}$ open subset of $\wt{E}$ associates the connected component of $z$ in $\wt{O}$, is measurable. We also refer to Chapter 2 \S1 of \cite{Math75} for further properties (results in \cite{Math75} are stated in terms of random closed sets). Recall assumptions (\ref{1.4}), (\ref{1.5}).

\begin{theorem}\label{theo1.1} (coupling and stochastic domination)

\smallskip\n
Consider $a \ge 0$, $\rho >0$, then 
\begin{align}
&\mbox{under $\wt{\IP}^G$ the random open set $\{\wt{\varphi} >a + \rho\}$ has the same law as the union of} \label{1.9}
\\[-1.2ex]
&
\mbox{connected components of $\{\wt{\varphi} >a\}$ that do not intersect $\wt{\cI}^{\ar}$ under $\wt{\IP}^G \otimes \wt{\IP}^I$,} \nonumber
\\[2ex]
&\mbox{under $\IP^G$, $\{\varphi >a + \rho\}$ has the same law as $\{\varphi > a\} \backslash \{$the union of connected}\label{1.10}
\\[-1.2ex]
&\mbox{components of $\{\wt{\varphi} > a\}$ that intersect $\wt{\cI}^{\ar}\}$ under $\wt{\IP}^G \otimes \wt{\IP}^I$} \nonumber
\end{align}

\n
(in the second occurrence in (\ref{1.10}) $\varphi$ is understood as the restriction of $\wt{\varphi}$ to $E$). Moreover, considering independent random interlacements and Gaussian free field on $E$, 
\begin{align}
&\mbox{$\{\varphi > a + \rho\}$ is stochastically dominated by $\{\varphi > a\} \backslash \{$the union of connected} \label{1.11}
\\[-1ex]
&\mbox{components in $E$ of $\cI^{\ar}$ for the conditional on $\varphi$ independent edge percol-}\nonumber
\\[-0.5ex]
& \mbox{ation on $E$ with success probability $1 - e^{-2c_{x,y}(\varphi_x-a)_+ (\varphi_y -a )_+}$ for $x \sim y$ in $E\}$}. \nonumber
\end{align}

%\begin{align}
%&\mbox{$\{\varphi > a + \rho\}$ is stochastically dominated by $\{\varphi > a\} \backslash \{$the union of connected} \label{1.11}
%\\[-1ex]
%&\mbox{components in $E$ of $\cI^{\ar}$ for the conditional on $\varphi$ independent edge percolation}\nonumber
%\\[-0.5ex]
%& \mbox{on $E$ with success probability $1 - e^{-2c_{x,y}(\varphi_x-a)_+ (\varphi_y -a )_+}$ for $x \sim y$ in $E\}$}. \nonumber
%\end{align}

\end{theorem}

\begin{proof}
We first prove (\ref{1.9}). We employ the coupling from Theorem 2.4 of \cite{Szni16} that we now recall. We consider $u > 0$ as well as independent $\wt{\varphi}$ (Gaussian free field on $\wt{E}$) and $(\wt{\ell}_{z,u})_{z \in \wt{E}}$ (field of local times of random interlacements at level $u$ on $\wt{E}$). We define (see (\ref{1.6}) for notation):
\begin{equation}\label{1.12}
\begin{array}{l}
\mbox{$\cJ =$ the union of the connected components of $\{2 \ell_{\cdot,u} + \wt{\varphi}_.^2 > 0\}$}
\\
\mbox{intersecting $\wt{\cI}^u$}.
\end{array}
\end{equation}
Then, see (2.26) of \cite{Szni16}, $\partial \cJ$ is a locally finite subset of $\wt{E}$, and we define
\begin{equation}\label{1.13}
\left\{ \begin{split}
\cC'_\infty & = \cJ \cup \partial \cJ, \;\mbox{and}
\\
\wt{\eta}_z & = \big(\sqrt{2u} - \sqrt{2 \wt{\ell}_{z,u} + \wt{\varphi}^2_z}\,\big) \;1\{z \in \cC'_\infty\} + (\wt{\varphi}_z + \sqrt{2u}) \; 1\{z \notin \cC'_\infty\}, \;\mbox{for $z \in \wt{E}$}.
\end{split}\right.
\end{equation}

\n
Then, Theorem 2.4 of \cite{Szni16} states that $\wt{\eta}_z$ is a continuous function of $z \in \wt{E}$, taking the value $\sqrt{2u}$ on $\partial \cJ$, and $(\wt{\eta}_z)_{z \in \wt{E}}$ is distributed as the Gaussian free field on $\wt{E}$. Thus, considering the random open set where the Gaussian free field is above $\sqrt{2u}$, one has
\begin{equation}\label{1.14}
\begin{array}{l}
\mbox{$\{\wt{\varphi} >\sqrt{2u}\} \stackrel{\rm law}{=} \{\wt{\eta} > \sqrt{2u}\} \stackrel{(\ref{1.13})}{=}$ the union of connected components of}
\\
\mbox{$\{\wt{\varphi} >0 \}$ that do not intersect $\wt{\cI}^u$}.
\end{array}
\end{equation}

\n
We will use (\ref{1.14}) with the special choices $u = \frac{1}{2}\, (a + \rho)^2$ and $u = \frac{1}{2} \,a^2$ (when $a > 0$). Thus, if $\wt{\cI}_1\, \!\!\! ^{\frac{a^2}{2}}$(understood as the empty set if $a=0$) and $\wt{\cI}_2\, \!\!\! ^{\ar}$ are independent realizations of the random interlacements on $\wt{E}$ at respective levels $\frac{a^2}{2}$ and $\ar$, then $\wt{\cI}_1\, \!\!\! ^{\frac{a^2}{2}} \cup \wt{\cI}_2\, \!\!\! ^{\ar}$ is distributed as $\wt{\cI}^{\frac{1}{2} \,(a + \rho)^2}$. By (\ref{1.14}) with $u = \frac{1}{2} \, (a + \rho)^2$, we find that:
\begin{equation}\label{1.15}
\begin{array}{rl}
\{\wt{\varphi} >a + \rho\}   \stackrel{\rm law}{=}  & \mbox{the union of connected components of 
$\{\wt{\varphi} > 0\}$}
\\[-0.5ex]
&  \mbox{that do not intersect $\wt{\cI}_1\, \!\!\! ^{\frac{a^2}{2}}$ or $\wt{\cI}_2^{\ar}$}
\\[2ex]
 = \;  &\mbox{the union of connected components of $\{\wt{\varphi} >0\}$}
\\[-1ex]
& \mbox{not intersecting $\wt{\cI}_1\, \!\!\! ^{\frac{a^2}{2}}$ that do not intersect $\wt{\cI}_2^{\ar}$}
\\[2ex]
&\hspace{-1.6cm} \underset{(\ref{1.14}) \, \rm with \, u = \frac{a^2}{2}}{\stackrel{\rm law}{=}} \mbox{the union of connected components of $\{\wt{\varphi} >a\}$}
\\[-2ex]
&\qquad  \!\! \mbox{that do not intersect $\wt{\cI}^{\ar}$}
\end{array}
\end{equation}

\n
(in the last step (\ref{1.14}) is only needed when $a > 0$). This proves (\ref{1.9}).

\smallskip
Taking the trace on $E$ of the identity (\ref{1.9}) immediately yields (\ref{1.10}). We now turn to the proof of (\ref{1.11}). By (\ref{1.10}) we know that $\{\varphi > a + \rho\}$ has the same law as 
$\{\varphi > a\} \backslash \{$the union of connected components of $\{\wt{\varphi} > a\}$ that intersect $\wt{\cI}^{\ar}\} \subseteq \{\varphi > a\}  \backslash \{$the union of connected components of $\{\wt{\varphi} > a\}$ that intersect $\cI^{\ar}\}$ (where $\cI^{\ar}$ stands for the intersection of $\wt{\cI}^{\ar}$ with $E$). Observe that the probability that a bridge in length $(2c_{x,y})^{-1}$, with boundary values $\varphi_x$ and $\varphi_y$, of a variance $2$ Brownian motion does not get below level $a$ equals $1-e^{-2c_{x,y}(\varphi_x-a)_+ (\varphi_y - a)_+}$, see for instance Lemma 2.1 of \cite{DrewPrevRodr}. Hence, by (\ref{1.3}) the above subset of $E$ is distributed as
\begin{align*}
\{\varphi > a\} \backslash \{& \mbox{the union of connected components of $\cI^{\ar}$ for the conditional on $\varphi$ }
\\[-0.5ex]
&\mbox{independent edge percolation on $E$ with success probability}
\\[-0.5ex]
&\mbox{$1-e^{-2c_{x,y}(\varphi_x-a)_+ (\varphi_y - a)_+}$ for each edge $x \sim y$ in $E\}$}.
\end{align*}

\n
This concludes the proof of (\ref{1.11}) and hence of Theorem \ref{theo1.1}.
\end{proof}

We will now give a simple minded consequence of Theorem \ref{theo1.1} showing the monotonicity property of the function $\tau$ introduced in (\ref{0.1}) along a family of arcs of parabolas in the upper quadrant. We recall the assumptions (\ref{1.4}), (\ref{1.5}). 

\begin{corollary}\label{cor1.2} (monotonicity property of $\tau$)

\smallskip\n
The function $\tau$ is non-increasing in $(u,a)$. However,
\begin{equation}\label{1.16}
\mbox{for all $h \ge 0$, the map $u \in [0,\frac{h^2}{2}] \rightarrow \tau\big(u, \sqrt{h^2 - 2u} \,\big) \in \IR_+$ is non-decreasing.}
\end{equation}
\end{corollary}

\begin{proof}
The first statement is plain from (\ref{0.1}). As for the second claim, consider $u \ge 0$, $a \ge 0$, $\rho >0$ as well as independent realizations $\cV^u, \cV_1^{\ar}$, and $\varphi$ of the respective vacant sets at levels $u$ and $\ar$ of the random interlacements on $E$, and of the Gaussian free field on $E$. By (\ref{1.11}) we find that
\begin{equation}\label{1.17}
\begin{array}{l}
\mbox{$\cV^u \cap \{\varphi > a + \rho\}$ is stochastically dominated by}
\\
\cV^u \cap \{\varphi > a\} \cap \cV_1^{\ar} \stackrel{\rm law}{=} \cV^{u + \ar} \cap \{\varphi >a\}.
\end{array}
\end{equation}

\n
Keeping in mind the definition of $\tau$ in (\ref{0.1}), this yields the inequality 
\begin{equation}\label{1.18}
\tau(u, a+ \rho) \le \tau\big(u + \textstyle \ar, a\big), \; \mbox{for $u,a, \rho \ge 0$}.
\end{equation}

\n
Now, given $h \ge 0$, $0 \le u' \le u \le \frac{h^2}{2}$, we set $a = \sqrt{h^2 - 2u}$, $a' = \sqrt{h^2 - 2u'} = a + \rho \ge a$, so that $\frac{h^2}{2} = u + \frac{a^2}{2} = u' + \frac{a'^2}{2} = u' + \frac{a^2}{2} + \ar$. It now follows from (\ref{1.18}) that
\begin{equation*}
\tau \big(u', \sqrt{h^2-2u'}\,\big) = \tau(u',a + \rho) \stackrel{(\ref{1.18})}{\le} \tau \big(u' + \textstyle \ar, a\big) = \big(u, \sqrt{h^2-2u}\,\big).
\end{equation*}
This proves our claim (\ref{1.16}).
\end{proof}

\begin{remark}\label{rem1.3} \rm 1) Letting $S_n(x_0)$ stand for the sphere in $E$ with center $x_0$ and radius $n$ for the graph distance on $E$, one can consider
\begin{equation}\label{1.19}
\mbox{$\tau_n(u,a) = \IP^I \otimes \IP^G \big[x_0 \overset{\cV^u \cap \{\varphi > a\}}{\mbox{\large $\longleftrightarrow$}} \;S_n(x_0)\big]$,  for $u \ge 0$, $a \in \IR$, $n \ge 1$.}
\end{equation}

\n
Then $\lim_{n} \tau_n(u,a) = \tau(u,a)$ and the same proof as in Corollary \ref{cor1.2} yields that
\begin{equation}\label{1.20}
\mbox{$\tau_n(u,a + \rho) \le \tau_n\big(u + \ar, a\big)$, for $u,a, \rho \ge 0, n \ge 1$},
\end{equation}

\n
which formally leads to the differential inequality in the spirit of \cite{AizeGrim91},
\begin{equation}\label{1.21}
\frac{1}{a} \; \frac{\partial \tau_n}{\partial a} \le \frac{\partial \tau_n}{\partial  u} \; (\le 0), \;\mbox{for $u,a > 0$}.
\end{equation}
The monotonicity property
\begin{equation}\label{1.22}
\begin{array}{l}
\mbox{for all $h \ge 0$, $n \ge 1$ the map $u \in \big[0,\frac{h^2}{2}\big] \rightarrow \tau_n\big(u, \sqrt{h^2 - 2u}\,\big) \ge 0$ is}
\\
\mbox{non-decreasing},
\end{array}
\end{equation}

\n
holds as well and can be viewed as an integrated version of (\ref{1.21}). It would be of interest to know whether some differential inequality can also be established when $a$ is negative.

\medskip\n
2) The positivity or the vanishing of $\tau(u,a)$ in (\ref{0.1}) does not depend on the choice of the base point $x_0$ in $E$. This fact can be seen with the successive application of the Harris-FKG inequality satisfied by $\cV^u$, see Theorem 3.1 of \cite{Teix09b}, followed by the Harris-FKG inequality for the Gaussian free field (and the use of Fubini's theorem).

\medskip\n
3) The above Corollary \ref{cor1.2} contains interesting information on the loci where $\tau$ is positive or vanishes. For instance, it implies that (see Figure 1 for the case of the regular tree):
\begin{equation}\label{1.23}
\mbox{$\tau >0$ on} \; \big\{(u,a) \in \IR_+ \times \IR_+ ; u + \textstyle \frac{a^2}{2} < \frac{h^2_*}{2}\big\}
\end{equation}
(since $\tau(u,a) \ge \tau (0, \sqrt{2u + a^2} \,) > 0$ on that set), and that
\begin{equation}\label{1.24}
\mbox{$\tau =0$ on} \; \big\{(u,a) \in \IR_+ \times \IR_+ ; u + \textstyle \frac{a^2}{2} > u_*\big\}
\end{equation}
(since $\tau(u,a) \le \tau (u + \frac{a^2}{2},0) = 0$ on that set).

\smallskip
Incidentally, (\ref{1.23}), (\ref{1.24}) yield some sufficient conditions to prove that $h_* < \sqrt{2u_*}$:
%
%incidentally provides some sufficient conditions to prove that $h_* < \sqrt{2u_*}$, as we now explain. For instance,
\begin{equation}\label{1.25}
\mbox{if $\tau(u,a) > 0$ for some $u,a \ge 0$ with $u + \frac{a^2}{2} > \frac{h^2_*}{2}$, then $h_* < \sqrt{2u_*}$}
\end{equation}
(otherwise (\ref{1.24}) would be contradicted), and
%
%Indeed, (\ref{1.16}) then implies that $\tau (u + \frac{a^2}{2}, 0) > 0) > 0$, so that $\cV^{u + \frac{a^2}{2}}$ percolates and hence $u_* \ge u + \frac{a^2}{2} > \frac{h^2_*}{2}$. Alternatively,
\begin{equation}\label{1.26}
\mbox{if $\tau(u,a) = 0$ for some $u,a \ge 0$ with $u + \frac{a^2}{2} < u_*$, then $h_* < \sqrt{2u_*}$}
\end{equation}
(otherwise (\ref{1.23}) would be contradicted). The above claim (\ref{1.24}) for the choice $a=0$, was for instance used in the proof of Theorem 3.4 of \cite{AbacSzni18}, to show that $h_* < \sqrt{2u_*}$ for a broad enough class of transient trees. One may wonder whether (\ref{1.23}) or (\ref{1.24}) could be helpful to establish $h_* < \sqrt{2u_*}$ in the case of $\IZ^d$, $d \ge 3$. \hfill $\square$
\end{remark}

\section{An application to the regular tree}
\setcounter{equation}{0}

We now consider the case when $T$ is a $(d+1)$-regular tree endowed with unit weights (we assume $d \ge 2$, and $T$ plays the role of $E$ in the previous section). We will now apply the stochastic domination (\ref{1.11}) of Theorem \ref{theo1.1} to prove the key Theorem \ref{theo2.1} and obtain as a consequence Corollary \ref{cor2.2} that proves the strict monotonicity property (\ref{0.9}) of the function $\lambda$ from (\ref{0.8}). In the next section Corollary \ref{cor2.2} will provide meaningful information about the region where $\tau$ vanishes or is positive, see Corollary \ref{cor3.2}.

We keep the notation from (\ref{0.4}) - (\ref{0.8}). One knows from Proposition 3.1 of \cite{Szni16} that
\begin{equation}\label{2.1}
\begin{array}{l}
\mbox{for all $h \in \IR$, $\lambda_h$ is a simple eigenvalue of $L_h$ and there is a unique}
\\
\mbox{function $\chi_h \ge 0$, with unit $L^2(\nu)$-norm, continuous and positive on}
\\
\mbox{$[h, + \infty)$, vanishing on $(-\infty,h)$, which is a corresponding eigenfunction.}
\\
\mbox{Moreover, $\chi_h$ belongs to all $L^p(\nu)$, $p \ge 1$}.
\end{array}
\end{equation}
The main result of this section is the following

\begin{theorem}\label{theo2.1}
For $a \ge 0$ and $\rho > 0$ one has
\begin{equation}\label{2.2}
\lambda_{a + \rho} < \lambda_a \,e^{-(\ar)\frac{(d-1)^2}{d}} .
\end{equation}
\end{theorem}

The inequality (\ref{2.2}) is a strong reinforcement of the inequality $\lambda_\rho \le \lambda_0 \,e^{-\frac{\rho^2}{2} \;\frac{(d-1)^2}{d}}$, for all $\rho \ge 0$, established in Theorem 4.3 of \cite{Szni16} (with the help of the coupling constructed in that article and recalled below (\ref{1.13})). Our proof of (\ref{2.2}) will be based on the stochastic domination (\ref{1.11}) of Theorem \ref{theo1.1} (which also relies on the above mentioned coupling). A direct analytical proof of (\ref{2.2}) is unknown to the author. Let us first state a consequence of Theorem \ref{theo2.1}.

\begin{corollary}\label{cor2.2} (strict monotonicity property of $\lambda$)

\smallskip\n
The function $\lambda$ from (\ref{0.8}) is strictly decreasing in $(u,a)$. However, for any $h > 0$, 
\begin{equation}\label{2.3}
\begin{array}{l}
\mbox{the map $u \in [0,\frac{h^2}{2}] \longrightarrow \lambda (u, \sqrt{h^2 - 2u})$ is strictly increasing and}
\\
\mbox{induces an homeomorphism between $[0,\frac{h^2}{2}]$ and $[\lambda_h, \lambda_0 \,e^{-\frac{h^2}{2} \,\frac{(d-1)^2}{d}}]$}.
\end{array}
\end{equation}
\end{corollary}

\smallskip\n
{\it Proof of Corollary \ref{cor2.2} (assuming Theorem \ref{theo2.1}).} The first statement is plain from (\ref{0.8}), (\ref{0.6}). For the second statement (\ref{2.3}), note that the map in (\ref{2.3}) is continuous, and we only need to show that it is strictly increasing. Consider $u' < u$ in $[0,\frac{h^2}{2}]$ and as below (\ref{1.8}) set $a = \sqrt{h^2 - 2u}$, $a' = \sqrt{h^2 - 2u'}$, and $\rho' = a' - a > 0$, so that $u + \frac{a^2}{2} = u' + \frac{a'^2}{2} = \frac{h^2}{2}$. Then we have
\begin{equation}\label{2.4}
\begin{split}
\lambda(u', a')  = &\; \lambda_{a + \rho} \,e^{-u' \frac{(d-1)^2}{d}} \stackrel{(\ref{2.2})}{<} \lambda_a \,e^{-(\ar + u') \frac{(d-1)^2}{d}}= \lambda_a\, e^{-(\frac{a'^2}{2} + u' - \frac{a^2}{2})\frac{(d-1)^2}{d}}
\\
=&\;\lambda_a \, e^{-(\frac{h^2}{2} - \frac{a^2}{2})\frac{(d-1)^2}{d}} = \lambda (u,a).
\end{split}
\end{equation}
This proves our claim (\ref{2.3}). \hfill $\square$

\medskip\n
{\it Proof of Theorem \ref{theo2.1}}. We consider an arbitrary base point $x_0$ in $T$, $n \ge 1$ and $x$ in $T$ at graph-distance $n$ from $x_0$. We denote by $[x_0,x]$ the geodesic segment in $T$ between $x_0$ and $x$. We consider $a \ge 0$, $\rho > 0$, and will bound first from below, and then from above, the quantity $\IP^G[x_0 \overset{\varphi > a + \rho}{\mbox{\large $\longleftrightarrow$}} x]$. Our claim (\ref{2.2}) will follow after letting $n$ tend to infinity.

We begin with the lower bound. An iterated application of the Markov property, see (3.13) of \cite{Szni16}, yields (with $ \langle \cdot, \cdot\rangle_\nu$ the scalar product in $L^2(\nu)$):
\begin{equation}\label{2.5}
\begin{split}
\IP^G \big[x_0 \overset{\varphi > a + \rho}{\mbox{\large $\longleftrightarrow$}} x\big] & = \IP^G\big[\varphi(y) > a + \rho \;\mbox{for all} \; y \in [x_0,x]\big] = \big\langle 1, \big(\textstyle\frac{1}{d} \;L_{a + \rho}\big)^n 1\big\rangle_\nu
\\
&\hspace{-1cm}\stackrel{\rm spectral \, expansion}{\ge} \big(\textstyle\frac{1}{d} \;\lambda_{a+\rho}\big)^n \langle \chi_{a + \rho}, 1\rangle^2_\nu.
\end{split}
\end{equation}

\n
We now turn to the upper bound. We denote by $\wt{\IP}$ the probability jointly governing $\varphi$, the Gaussian free field on $T$, and conditionally on $\varphi$ an independent Bernoulli edge percolation on $T$, where each edge $x \sim y$ in $T$ is open with probability $1 - e^{-2(\varphi_x - a)_+(\varphi_y - a)_+}$. The stochastic domination (\ref{1.11}) crucially yields the upper bound:
\begin{equation}\label{2.6}
\begin{array}{l}
\IP^G \big[x_0 \overset{\varphi > a + \rho}{\mbox{\large $\longleftrightarrow$}} x\big] \le \wt{\IP} \otimes \IP^I  \big[x_0 \overset{\varphi > a}{\mbox{\large $\longleftrightarrow$}} x , [x_0,x] \subseteq \cV^{\ar} \;\mbox{and for all $z \in \cI^{\ar}$}
\\[0.5ex]
\mbox{neighboring $[x_0,x]$, the edge between $z$ and $[x_0,x]$ is closed$\big]$}.
\end{array}
\end{equation}

\n
Recall $\sigma^2$ from (\ref{0.4}). It is convenient to introduce the notation
\begin{equation}\label{2.7}
p_0 = e^{-(\ar) \sigma^{-2}}, \, p= e^{-(\ar)\frac{(d-1)^2}{d}} .
\end{equation}

\n
These two quantities respectively equal the probability that $x_0$ belongs to $\cV^{\ar}$, and for $x \not= x_0$ that there is no trajectory in the interlacement at level $\ar$ having $x$ as its point of minimum distance to $x_0$. All these events are known to be independent, see Section 5 of \cite{Teix09b}, in particular (5.7), (5.9) of this reference (note that $d+1$ plays the role of $d$ in \cite{Teix09b} and here the weights are unit). The right-hand side of (\ref{2.6}) then equals
\begin{equation}\label{2.8}
p_0 \,p^n \,\IE^G\big[\varphi(y) > a \;\mbox{for all $y \in [x_0,x]$} , \textstyle \prod\limits_{y \in [x_0,x]}  \prod\limits_{\mbox{\scriptsize $z \sim y$}, \atop z \notin [x_0,x]} \!\! \big(p+ (1-p) \,e^{-2(\varphi_z - a)_+ (\varphi_y -a)_+}\big)\big].
\end{equation}

\n
With $Q_t, t \ge 0$, from below (\ref{0.5}), we introduce the measurable function from $\IR$ to $(0,1)$:
\begin{equation}\label{2.9}
V(b) = p+(1-p) \,Q_{\frac{1}{d}} (e^{-2(\cdot\, - a)_+(b-a)_+}) (b).
\end{equation}

\n
By the Markov property we then find that the expression in (\ref{2.8}) is bounded above by:
\begin{equation}\label{2.10}
p_0\,p^n \, \IE^G\big[\varphi(y) > a \;\mbox{for all} \;y \in [x_0,x], \textstyle\prod\limits_{y \in [x_0,x]} V(\varphi_y)\big].
\end{equation}

\n
We now introduce the self-adjoint Hilbert-Schmidt operator on $L^2(\nu)$ (see (\ref{0.5})):
\begin{equation}\label{2.11}
\wt{L}_h = \sqrt{V} \, L_h \sqrt{V}, \;\mbox{for $h \in \IR$} ,
\end{equation}
and denote its operator norm by
\begin{equation}\label{2.12}
\wt{\lambda}_h = \| \wt{L}_h\|_{L^2(\nu) \rightarrow L^2(\nu)},  \;\mbox{for $h \in \IR$} .
\end{equation}

\n
Keeping in mind that (\ref{2.8}) bounds $\IP^G [x_0 \overset{\varphi > a + \rho}{\mbox{\large $\longleftrightarrow$}} x]$ from above, we obtain after repeated application of the Markov property as in (\ref{2.5}), that
\begin{equation}\label{2.13}
\IP^G [x_0 \overset{\varphi > a + \rho}{\mbox{\large $\longleftrightarrow$}} x] \le p_0 \, p^n \big\langle \sqrt{V}, \big(\textstyle\frac{1}{d} \;\wt{L}_a\big)^n \sqrt{V} \big\rangle_\nu \le p_0 \, p^n \big(\textstyle \frac{1}{d} \;\wt{\lambda}_a\big)^n \|\sqrt{V} \|^2_{L^2(\nu)}.
\end{equation}

\n
Comparing the lower bound (\ref{2.5}) and the upper bound (\ref{2.13}) (note that $\langle \chi_{a+\rho}, 1\rangle_\nu > 0$ by (\ref{2.1})), we find after taking $n$-th roots and letting $n$ tend to infinity that
\begin{equation}\label{2.14}
\lambda_{a+ \rho} \le \wt{\lambda}_a \, p \stackrel{(\ref{2.7})}{=} \wt{\lambda}_a \,e^{-(\ar)\frac{(d-1)^2}{d}}.
\end{equation}

\n
The proof of (\ref{2.2}), i.e.~of Theorem \ref{theo2.1}, will thus be completed once we show that
\begin{equation}\label{2.15}
\wt{\lambda}_a < \lambda_a .
\end{equation}

\n
We now prove (\ref{2.15}). We denote by $\gamma_a < \lambda_a$ the operator norm of $L_a$ restricted to the orthogonal subspace of $\chi_a$ in $L^2(\nu)$, see (\ref{2.1}). Thus, for $f \in L^2(\nu)$ with unit norm, $h = \sqrt{V} f - \langle \chi_a, \sqrt{V}  f\rangle_\nu \, \chi_a$ is orthogonal to $\chi_a$ in $L^2(\nu)$, and we have
\begin{equation}\label{2.16}
\begin{split}
\langle f, \wt{L}_a \,f\rangle_\nu  &= \;  \langle \sqrt{V} f, L_a \sqrt{V} \, f\rangle_\nu \stackrel{(\ref{2.1}), (\ref{0.5})}{=} \lambda_a \,\langle \chi_a, \sqrt{V}f \, \rangle^2_\nu + \langle h, L_a h\rangle_\nu
\\
& \le  \lambda_a \langle \chi_a, \sqrt{V}f\rangle^2_\nu + \gamma_a \|h\|^2_{L^2(\nu)} = (\lambda_a - \gamma_a) \langle \chi_a, \sqrt{V}f\rangle^2_\nu + \gamma_a \|\sqrt{V} f\|^2_{L^2(\nu)}
\\
& \hspace{-1cm} \underset{0 < V < 1}{\stackrel{\mbox{\scriptsize Cauchy-Schwarz}}{\le}} (\lambda_a - \gamma_a) \,\langle \chi^2_a,V\rangle_\nu + \gamma_a \stackrel{0 < V < 1}{<} \lambda_a.
\end{split}
\end{equation}

\n
Taking a supremum over $f$ of unit norm in $L^2(\nu)$ shows that $\wt{\lambda}_a$ is at most equal to $(\lambda_a - \gamma_a) \langle \chi^2_a, V\rangle_\nu + \gamma_a$, and (\ref{2.15}) follows. This concludes the proof of Theorem \ref{theo2.1}. \hfill $\square$

\section{Application to vacant set level set percolation on a regular tree}
\setcounter{equation}{0}

In this section we keep the set-up of the previous section. Theorem \ref{theo3.1} below provides the link between the function $\lambda(u,a)$ from (\ref{0.8}) and the positivity or vanishing of the percolation function $\tau(u,a)$ from (\ref{0.1}). In particular it highlights the role of the curve $\lambda(u,a) = 1$ as a ``critical line'' for the vacant set level set percolation. The strict monotonicity property established in Corollary \ref{cor2.2} is then used in Corollary \ref{cor3.2} to infer some significant properties of the regions where $\tau$ is positive or vanishes.

Incidentally, in the present case of the $(d+1)$-regular tree with unit weights (and $d \ge 2$), the function $\tau$ itself (and not merely its positivity or vanishing) is independent of the choice of the base point $x_0$ (see also Remark \ref{rem1.3} 2)).

\begin{theorem}\label{theo3.1} (link between $\lambda$ and $\tau$) 

\smallskip\n
For $u \ge 0$, $a \in \IR$,
\begin{align}
& \mbox{if $\lambda(u,a) < 1$, then $\tau(u,a) = 0$}, \label{3.1}
\\[1ex]
& \mbox{if $\lambda(u,a) > 1$, then $\tau(u,a) > 0$}. \label{3.2}
\end{align}
\end{theorem}

\begin{proof}
We first prove (\ref{3.1}). For $n \ge 1$, and $x$ in $T$ at graph distance $n$ from the base point $x_0$ in $T$, we obtain by a similar calculation as in (\ref{2.5}) (see (\ref{2.7}) for notation):
\begin{equation}\label{3.3}
\IP^I \otimes \IP^G \big[x_0 \overset{\cV^u \cap \{\varphi > a\}}{\mbox{\large $\longleftrightarrow$}} x\big] = p_0 \,p^n \IP^G \big[x_0 \overset{\varphi > a}{\mbox{\large $\longleftrightarrow$}} x\big] \stackrel{\rm as \; in \,(\ref{2.5})}{\le} p_0 \,p^n \big(\textstyle \frac{1}{d} \;\lambda_a\big)^n.
\end{equation}
Thus with $\tau_n$ as in (\ref{1.19}) we have
\begin{equation}\label{3.4}
\tau_n (u,a) \le (d+1) \,d^{n-1} p_0 \, p^n \big(\textstyle \frac{1}{d} \;\lambda_a\big)^n \underset{(\ref{0.8})}{\stackrel{(\ref{2.7})}{=}} \;\frac{d+1}{d} \;p_0 \, \lambda(u,a)^n.
\end{equation}

\n
Letting $n$ tend to infinity, we see that $\lambda(u,a) < 1$ implies $\tau(u,a) = 0$, and (\ref{3.1}) follows.

\smallskip
We now turn to the proof of (\ref{3.2}). For $n \ge 1, u \ge 0, a \in \IR$, we introduce (see (\ref{1.19}) and (\ref{2.1}) for notation)
\begin{equation}\label{3.5}
M_n = \lambda(u,a)^{-n} \textstyle \sum\limits_{x \in S_n(x_0)} 1 \big\{x_0 \overset{\cV^u \cap \{\varphi > a\}}{\mbox{\large $\longleftrightarrow$}} x\big\} \;\chi_a (\varphi_x) \;(\ge 0).
\end{equation}

\n
We will show that under $\IP^I \otimes \IP^G$ the first moment of $M_n$ is uniformly positive (and in fact constant), and that the second moment is uniformly bounded. We begin with the calculation of the first moment and recall the notation (\ref{2.7}). We have
\begin{equation}\label{3.6}
\IE^I \otimes \IE^G [M_n] = \lambda(u,a)^{-n} p_0\,p^n \textstyle  \sum\limits_{x \in S_n(x_0)} \IE^G \big[x_0 \overset{\varphi > a}{\mbox{\large $\longleftrightarrow$}} x, \chi_a(\varphi_x)\big] .
\end{equation}

\n
The last expectation is the same for all $x \in S_n(x_0)$. Using the Markov property repeatedly as in (\ref{2.5}), we find that
\begin{equation}\label{3.7}
\begin{split}
\IE^I \otimes \IE^G [M_n] & = (d+1) \, d^{n-1} \lambda(u,a)^{-n} p_0 \,p^n \big\langle1, \big(\textstyle \frac{1}{d} \;L_a\big)^n \chi_a \big\rangle_\nu
\\
&\hspace{-1ex} \stackrel{(\ref{2.1})}{=} \textstyle \frac{(d+1)}{d} \;\lambda(u,a)^{-n} p_0\,p^n \,\lambda^n_a \,\langle 1,\chi_a\rangle_\nu = \frac{(d+1)}{d} \;p_0  \,\langle 1, \chi_a\rangle_\nu \stackrel{\rm def}{=} A \stackrel{(\ref{2.1})}{>} 0.
\end{split}
\end{equation}

\n
We now turn to the control of the second moment of $M_n$. We have
\begin{equation}\label{3.8}
\begin{array}{l}
\IE^I \otimes \IE^G [M^2_n] = 
\\
\lambda(u,a)^{-2n} \sum\limits_{x,y \in S_n(x_0)} \IE^I \otimes \IE^G  \big[x_0 \overset{\cV^u \cap \{\varphi > a\}}{\mbox{\large $\longleftrightarrow$}} x, x_0 \overset{\cV^u \cap \{\varphi > a\}}{\mbox{\large $\longleftrightarrow$}} y, \chi_a (\varphi_x) \chi_a(\varphi_y)\big].
\end{array}
\end{equation}

\n
Given $x,y \in S_n(x_0)$, we write $x \wedge y$ for the point $z \in T$ such that $[x_0,z] = [x_0,x] \cap [x_0,y]$. We note that when $x \wedge y$ is at distance $k  (\in \{0, \dots,n\})$ from $x_0$, the expectation under the sum in (\ref{3.8}) equals:
\begin{equation}\label{3.9}
\begin{array}{l}
p_0 \,p^k \,p^{2(n-k)} \,\IE^G \big[x_0 \overset{\varphi > a}{\mbox{\large $\longleftrightarrow$}} x, x_0  \overset{\varphi > a}{\mbox{\large $\longleftrightarrow$}} y, \chi_a(\varphi_x) \,\chi_a(\varphi_y)\big] \underset{(\ref{2.1})}{\stackrel{\rm Markov \; property}{=}}
\\
p_0 \,p^{2n-k} \big(\textstyle \frac{1}{d} \;\lambda_a\big)^{2(n-k)} \IE^G \big[x_0 \overset{\varphi > a}{\mbox{\large $\longleftrightarrow$}}x \wedge y, \chi_a(\varphi_{x \wedge y})^2\big].
\end{array}
\end{equation}

\n
Applying the Markov property repeatedly once again, the above quantity equals
\begin{equation}\label{3.10}
\begin{array}{l}
p_0 \,p^{2n-k} \big(\textstyle \frac{1}{d} \;\lambda_a\big)^{2(n-k)}  \big\langle  \big(\textstyle \frac{1}{d} \;L_a\big)^k 1, \chi_a^2 \big\rangle_\nu \le
\\[1ex]
p_0 \,p^{2n-k} \big(\textstyle \frac{1}{d} \;\lambda_a\big)^{2n-k} \,\|\chi^2_a\|_{L^2(\nu)} \stackrel{(\ref{0.8}),(\ref{0.7})}{=} p_0 \big(\textstyle \frac{1}{d} \;\lambda(u,a)\big)^{2n-k} \| \chi^2_a\|_{L^2(\nu)}.
\end{array}
\end{equation}
Coming back to (\ref{3.8}) we thus find that
\begin{equation}\label{3.11}
\begin{split}
\IE^I \otimes \IE^G [M^2_n] & \le p_0 \,\|\chi_a^2\|_{L^2(\nu)} \textstyle \sum\limits^n_{k=0} \;\textstyle \sum\limits_{z \in S_k(x_0)} \;\sum\limits_{x,y \in S_n(x_0):x \wedge y = z} \big(\textstyle \frac{1}{d}\; \lambda(u,a)\big)^{-k} d^{-2n}
\\
& \le p_0 \,\|\chi_a^2\|_{L^2(\nu)}\textstyle  \sum\limits^n_{k=0} (d+1) \,d^{k-1} \,d^{2(n-k)} \,d^{k-2n} \,\lambda(u,a)^{-k}
\\
&\le \textstyle \frac{(d+1)}{d} \;p_0 \|\chi_a^2\|_{L^2(\nu)}  \sum\limits^\infty_{k=0} \lambda(u,a)^{-k} \stackrel{\rm def}{=} B < \infty
\end{split}
\end{equation}

\n
since $\lambda(u,a) > 1$ and $\chi_a$ belongs to all $L^p(\nu), p \ge 1$, by (\ref{2.1}). Note that $M_n = M_n \,1\{ 0 \,\overset{\cV^u \cap \{\varphi > a\}}{\mbox{\large $\longleftrightarrow$}} S_n(x_0)\}$. So, by the Cauchy-Schwarz Inequality we obtain (see (\ref{1.19}) for notation):
\begin{equation}\label{3.12}
A^2 \stackrel{(\ref{3.7})}{=} \IE^I \times \IE^G [M_n]^2 \le \tau_n(u,a) \,\IE^I \otimes \IE^G [M_n^2] \stackrel{(\ref{3.11})}{\le} \tau_n(u,a) \,B.
\end{equation}
Letting $n$ tend to infinity now yields
\begin{equation}\label{3.13}
\tau(u,a) \ge \textstyle \frac{A^2}{B} > 0.
\end{equation}

\n
This shows (\ref{3.2}) and concludes the proof of Theorem \ref{theo3.1}. 
\end{proof}

\smallskip
Thus, in view of Theorem \ref{theo3.1}, the curve $\lambda(u,a) = 1$ stands as a separation line between the region $\lambda(u,a) > 1$, where the vacant set level set percolation occurs and the region $\lambda(u,a) < 1$, where the vacant set level set percolation does not occur. With the help of Corollary \ref{cor2.2} we will gain further insight on the regions where $\tau$ is positive or vanishes, see also Figure 1 below.

\begin{corollary}\label{cor3.2} (see (\ref{0.3}), (\ref{0.7}) for notation)
\begin{align}
&\mbox{$\tau(u,a) > 0$ in a neighborhood in $\IR_+ \times \IR$ of $\big\{(u,a); u > 0, a \ge 0, u + \textstyle \frac{a^2}{2} = \frac{h^2_*}{2}\big\}$}, \label{3.14}
\\
&\mbox{$\tau(u,a) = 0$ in a neighborhood in $\IR_+ \times \IR$ of $\big\{(u,a); u \ge 0, a \ge 0, u + \textstyle \frac{a^2}{2} = u_*\big\}$}. \label{3.15}
\end{align}
\end{corollary}

\begin{proof}
The function $\lambda(u,a)$ from (\ref{0.8}) is continuous on $\IR_+ \times \IR$. By Corollary \ref{cor2.2}, when $(u,a)$ belongs to the set that appears in (\ref{3.14}), $\lambda(u,a) > \lambda(0,h_*) = \lambda_{h_*} = 1$. Hence, the strict inequality $\lambda(u,a) > 1$ holds in a neighborhood of this set, and (\ref{3.14}) now follows from (\ref{3.2}) of Theorem \ref{theo3.1}.

\smallskip
As for (\ref{3.15}), note that the maximal value reached by $\lambda(u,a)$ on the set that appears in (\ref{3.15}) is by Corollary \ref{cor2.2} equal to $\lambda(u_*,0) = \lambda_0 \,e^{-u_* \frac{(d-1)^2}{d}} \stackrel{(\ref{0.3})}{=} \textstyle \frac{\lambda_0}{d} \stackrel{(\ref{0.6})}{<} 1$.
\end{proof}

\psfrag{0}{$0$}
\psfrag{a}{$a$}
\psfrag{u}{$u$}
\psfrag{uu}{$u_*$}
\psfrag{h}{$h_*$}
\psfrag{2u}{$\sqrt{2u_*}$}
\psfrag{t}{$\tau > 0$}
\psfrag{t0}{$\tau = 0$}
\psfrag{h2}{$h^2_*/2$}
\begin{center}
\includegraphics[width=7cm]{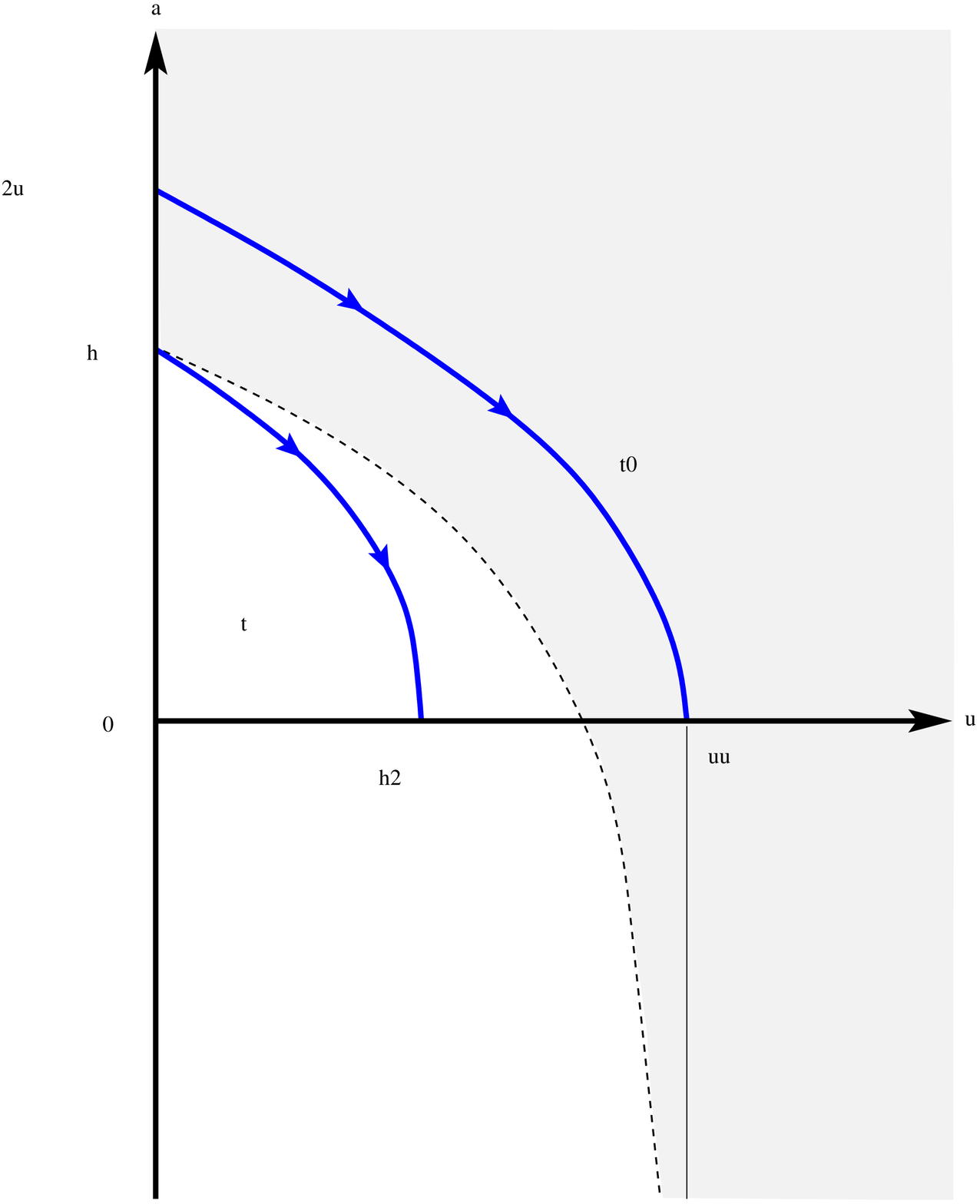}
\end{center}

\bigskip
\begin{center}
Fig.~1: The percolation diagram in the case of a regular tree.
\begin{tabular}{lp{15cm}}
& The dotted line corresponds to the ``critical line'' $\lambda(u,a) = 1$.
The two thick lines correspond to the arcs $(u, \sqrt{h^2 - 2u}$), $0 \le u \le \frac{h^2}{2}$,
with the respective choices $h = h_*$ and $h = \sqrt{2u_*}$. 
Along all such arcs $\tau$ is non-decreasing by Corollary 1.2.  The
regions $\tau = 0$ and $\tau > 1$ are depicted in accordance with Corollary \ref{cor3.2}
and the non-increasing property of $\tau$,  see Corollary \ref{cor1.2}.
\end{tabular}
\end{center}

\n
One can naturally wonder whether a similar looking percolation diagram holds as well in the case of the vacant set level set percolation on $\IZ^d$, $d \ge 3$. It is presently not known whether the two thick arcs in Figure 1 are distinct in this case (one knows that $0<h_* \le \sqrt{2 u_*} <  \infty$, and the positivity of $h_*$ for all  $d \ge 3$ has only recently been established in \cite{DrewPrevRodr}).

%\newpage

\end{document}